\documentclass[a4paper,9pt]{article}
\usepackage[a4paper, total={15cm, 23.7cm}]{geometry}
\usepackage{amssymb,amsthm,amsmath,amsfonts,color,tikz}
\usepackage{comment}
\usepackage[normalem]{ulem}

\newtheorem{theorem}{Theorem}

\newcommand{\R}{\mathbb{R}}
\newcommand{\N}{\mathbb{N}}

\def\cp{\mathrm{cap}\,}
\def\diam{\mathrm{diam}}

\newcommand{\Om}{\Omega}

\newcommand{\bp}{\begin{proof}}
\newcommand{\ep}{\end{proof}}

\begin{document}
\title{On some variational problems involving capacity, torsional rigidity, perimeter and measure}

\author{{M. van den Berg} \\
School of Mathematics, University of Bristol\\
Fry Building, Woodland Road\\
Bristol BS8 1UG, United Kingdom\\
\texttt{mamvdb@bristol.ac.uk}\\
\\
A. Malchiodi\\
Scuola Normale Superiore\\
Piazza dei Cavalieri 7\\
56126 Pisa, Italy\\
\texttt{andrea.malchiodi@sns.it}}
\date{24 January 2022}
\maketitle
\vskip1truecm\indent

\begin{abstract}\noindent
We investigate the existence of a maximiser among open, bounded, convex sets in $\R^d,\,d\ge 3$ for the product of torsional rigidity and Newtonian capacity (or logarithmic capacity if $d=2$), with constraints involving Lebesgue measure or a combination of Lebesgue measure and perimeter.
\end{abstract}


\section{Introduction \label{sec1}}

The classic treatise of G. P\'{o}lya, G. Szeg\"{o}, \cite{PSZ}, displays a wealth of isoperimetric inequalities involving quantities such as the principal eigenvalue of the Dirichlet Laplacian, moment of inertia, capacity, perimeter etc. Many of the inequalities involve just two quantities like torsional rigidity and measure (de Saint Venant's inequality, \eqref{e8} below),  principal Dirichlet eigenvalue and measure (Rayleigh Faber Krahn inequality), torsional rigidity and principal Dirichlet eigenvalue (Kohler-Jobin inequality). In each of these inequalities the ball is optimal, and their stability has been investigated in depth. See for example \cite{BDP}, and the references therein.
Examples of inequalities involving three quantities are numerous too. In this case it is convenient to display the various inequalities in a Blaschke-Santal\`o diagram. See for example \cite{MvdBBP} and \cite{LZ} for torsional rigidity, principal Dirichlet eigenvalue and measure, and \cite{FL} for perimeter, principal Dirichlet eigenvalue, and measure.

For example an inequality going back to P\'olya (5.4 in \cite{PSZ} for the planar case) asserts that for any open set with finite measure the product of torsional rigidity and principal Dirichlet eigenvalue is bounded by the Lebesgue measure. In \cite{MvdBFNT} it was shown that this inequality is sharp but that no optimal set exists. A more general functional was studied in  \cite{MvdBBP}. There the principal eigenvalue with the $q$-'th power of the torsional rigidity was studied among open sets with fixed Lebesgue measure. For $0<q<2/(d+2)$ the principal eigenvalue dominates the behaviour, and the ball is the minimiser.
On the other hand for $q>1$ the torsional rigidity dominates the behaviour, and the functional has a maximiser in the class of open, bounded, convex sets. Recently functionals involving powers of the perimeter and torsional rigidity have been studied under a fixed measure constraint, \cite{BBP1} and \cite{BBP2}. In order to guarantee well-posedness of the problem it suffices to consider the collection of non-empty, open, bounded and convex sets. If $d=2$ then the convexity constraint could be relaxed by considering non-empty, open, bounded Lipschitz sets with a topological constraint instead, \cite{BBP2}.

A second example involves Newtonian capacity (or logarithmic capacity if $d=2$), torsional rigidity and measure. Unlike measure or torsional rigidity, capacity is not additive on disjoint closed sets.
In order to obtain well-posed and/or non-trivial examples, the class of admissible sets is the collection of open, non-empty, bounded, convex sets in $\R^d$.

The torsion function for a non-empty open set $\Om\subset \R^d$ with finite Lebesgue measure $|\Om|$ is the solution of
\begin{equation*}
-\Delta u=1,\quad u\in H_0^1(\Om),
\end{equation*}
and is denoted by $u_\Om$. The torsional rigidity of $\Om$, or torsion for short, is denoted by
\begin{equation*}
T(\Om)=\|u_{\Om}\|_1,
\end{equation*}
where $\|\cdot\|_p,\,1\le p\le \infty$ denotes the usual $L^p$ norm. The torsion satisfies the scaling property
\begin{equation}\label{e5}
T(t\Om)=t^{d+2}T(\Om),\,t>0,
\end{equation}
where for any set $A\subset \R^d$, $tA=\{tx:x\in A\}$ is the homethety (scaling) of $A$ by a factor $t$.

Let $\cp(K)$ denote the Newtonian capacity of a compact set $K\subset\R^d$ if $d>2$ or the logarithmic capacity if $d=2$.
It follows directly from its definition (see \cite{L}) that for $d>2$,
\begin{equation}\label{e12}
\cp(tK)=t^{d-2}\cp(K),\quad t>0.
\end{equation}
Let
\begin{equation}\label{e18}
G(\Om)=\frac{T(\Om)\cp(\overline{\Om})}{|\Om|^2},
\end{equation}
where $\overline{A}$ denotes the closure of $A$.
By \eqref{e5} and \eqref{e12} we obtain that $G(\Om)$ is scaling invariant. It is easily seen that $G(\Om)$ is not bounded from above on the class of non-empty open sets with finite measure. See Theorem 2 (i) in \cite{vdBB}. On the other hand it was shown in Theorem  2(iii) of \cite{vdBB}
that the variational problem
\begin{equation}\label{e19}
\sup\{G(\Om) :\, \Om\textup{ non-empty, open, bounded and convex}\}
\end{equation}
has a maximiser for $d=3$, and that any such maximiser, denoted by $\Om^+$, satisfies
\begin{equation}\label{e20}
\frac{\diam(\Om^+)}{r(\Om^+)}\le 2\cdot3^8e^{3^7},
\end{equation}
where $\textup{diam}(\cdot)$ denotes diameter, and $r(\cdot)$ denotes inradius.

The proof of this result rests on the fact that if $d=3$ then $\lim_{n\rightarrow\infty}G(\Om_n)=0$ for a sequence $(\Om_n)$ of elongated ellipsoids.

\medspace

Below we outline our contribution to these and related problems.

\medspace

Section \ref{sec2} mainly concerns the analysis of the variational problem in \eqref{e19}.
In Theorem 1 below we obtain some characterisation in higher dimensions for  elements of a maximising sequence. In particular if $d=4$ we show that there is at most one direction of elongation direction for such sequences, and that the other three directions have comparable lengths. This however, is not sufficient to prove existence of a maximiser of \eqref{e19} for $d=4$.

In Theorem \ref{the2} we show that if we maximise $G(\Om)$ over the collection of ellipsoids then \eqref{e19} has the ball as the unique maximiser. We use this to improve the upper bound for $G(\Om)$ obtained in Theorem 2(ii) in \cite{vdBB}. We also obtain an expression for $G(\Om)$ in terms of its eccentricity in case $\Om$ is an ellipsoid. This in turn is used to prove that $G(\Om)$ has a maximiser if the class of admissible sets is well approximated by an ellipse containing $\Om$ (Theorem \ref{the3}).

\medspace

Section \ref{sec3} concerns the analysis of variational problems involving collections of planar sets $\Om\subset \R^2$ with corresponding logarithmic capacity $\cp(\Om)$.

It was shown in \cite{vdBB} that if $d=2$ and if $q\ge \frac12$, then the supremum for $T(\Omega)^q\cp(\overline{\Omega})$ among all open, planar, convex sets with fixed measure is finite (Theorem 4(i) in \cite{vdBB}). If $q>\frac12$ then the supremum is achieved
for some open convex set with the same measure. In Theorems \ref{the4} and \ref{the5} we present results for the critical case $q=\frac12$ corresponding to Theorems \ref{the2} and \ref{the3} respectively. For $d=2$ we let
\begin{equation}\label{e64}
H(\Omega)=\frac{T(\Omega)^{1/2}\cp(\overline{\Omega})}{|\Omega|^{3/2}}.
\end{equation}
Since logarithmic capacity scales as
\begin{equation*}
\cp(t\overline{\Omega})=t\cp(\overline{\Omega}),\,t>0,
\end{equation*}
we see that $H(\Om)$ is scaling invariant.

\medspace

Finally, Section \ref{sec4} concerns modifications of the functionals \eqref{e18} and \eqref{e64} involving
the perimeter.
It is a well known fact that many variational problems with a perimeter constraint are easier to handle than those with a measure constraint. See for example \cite{DPLMV}. If we maximise $G(\Om)$ or $H(\Om)$
over all open, convex sets with a fixed perimeter, then compactness in the Hausdorff metric establishes the existence of a corresponding maximiser. In Theorems \ref{the6} and \ref{the8} we obtain results which interpolate between the extremes of measure and perimeter constraints. It is shown that an arbitrary small (positive) power of the perimeter in the denominators of \eqref{e18} and \eqref{e64} is sufficient to establish the existence of maximisers. An estimate for the ratio of diameter and inradius is obtained in terms of that power. That ratio diverges as the power decreases to $0$.

\medspace

Below we recall some basic facts about torsional rigidity and Newtonian capacity.
It is convenient to extend $u_\Om$ to all of $\R^d$ by defining $u_\Om=0$ on $\R^d\setminus\Om$. It is well known that $u_\Om$ is non-negative, bounded, and monotone increasing with respect to $\Om$. That is if $\Omega_1$ is open and non-empty, $\Omega_2$ is open with finite Lebesgue measure, then
\begin{equation*}
\Om_1\subset\Om_2\Rightarrow u_{\Om_1}\le u_{\Om_2}.
\end{equation*}
 It follows that
\begin{equation}\label{e4}
\Om_1\subset\Om_2\Rightarrow T(\Om_1)\le T(\Om_2).
\end{equation}

If $E(a)$, with $a=(a_1,a_2,\dots,a_d)\in\R_+^d$, is the ellipsoid
\begin{equation*}
E(a)=\bigg\{x\in\R^d\ :\ \sum_{i=1}^d\frac{x_i^2}{a_i^2}<1\bigg\},
\end{equation*}
then
$$u_{E(a)}(x)=\frac12\bigg(\sum_{i=1}^d\frac{1}{a_i^2}\bigg)^{-1}\bigg(1-\sum_{i=1}^d\frac{x_i^2}{a_i^2}\bigg),$$
and
\begin{equation}\label{e7}
T(E(a))=\frac{\omega_d}{d+2}\bigg(\prod_{i=1}^d a_i\bigg) \bigg(\sum_{i=1}^d\frac{1}{a_i^2}\bigg)^{-1},
\end{equation}
where
$$\omega_d=\frac{\pi^{d/2}}{\Gamma((d+2)/2)}$$
is the Lebesgue measure of a ball $B_1$ with radius $1$ in $\R^d$. We put
\begin{equation}\label{e8e}
\tau_d=T(B_1)=\frac{\omega_d}{d(d+2)}.
\end{equation}
The de Saint-Venant inequality (see \cite{PSZ}) asserts that
\begin{equation}\label{e8}
T(\Om)\le T(\Om^*),
\end{equation}
where $\Om^*$ is any ball with $|\Om|=|\Om^*|$. It follows by scaling that
\begin{equation*}
\frac{T(\Om)}{|\Om|^{(d+2)/2}}\le\frac{\tau_d}{\omega_d^{(d+2)/d}}=\frac{1}{d(d+2)\omega_d^{2/d}}.
\end{equation*}

For various equivalent definitions of the Newtonian capacity of a compact set in $\R^d,\,d>2,$ we refer to \cite{L}.
If $K_1\ne \emptyset$ and $K_2$ are compact sets, then
\begin{equation}\label{e11}
K_1\subset K_2\Rightarrow \cp(K_1)\le\cp(K_2).
\end{equation}
It was reported in \cite{IMcK} p.260 that the Newtonian capacity of an ellipsoid was computed in volume 8, p.30 in \cite{GC}. The formula there is for a three-dimensional ellipsoid, and is given in terms of an elliptic integral. It extends to all $d\ge 3$, and reads
\begin{equation}\label{e13}
\cp\big(\overline{E(a)}\big)=\frac{\kappa_d}{\frac{d}{2}-1}\mathfrak{e}(a)^{-1},
\end{equation}
with
\begin{equation}\label{e14}
\mathfrak{e}(a)=\int_0^{\infty}dt\,\bigg(\prod_{i=1}^d\big(a_i^2+t\big)\bigg)^{-1/2},
\end{equation}
and where
\begin{equation}\label{e15}
\kappa_d=\cp(\overline{B_1})=\frac{4\pi^{d/2}}{\Gamma((d-2)/2)}.
\end{equation}
For further references and applications of \eqref{e13}-\eqref{e14} see \cite{T} and \cite{FGP}.

The isoperimetric inequality for Newtonian capacity (see \cite{PSZ}) asserts that for all non-empty compact sets $K\subset \R^d$, $d\ge 3$,
\begin{equation*}
\cp(K)\ge\cp(K^*),
\end{equation*}
where $K^*$ is any closed ball with $|K|=|K^*|$. It follows by scaling that
\begin{equation*}
\frac{\cp(K)}{|K|^{(d-2)/d}}\ge \frac{\kappa_d}{\omega_d^{(d-2)/d}}=d(d-2)\omega_d^{2/d}.
\end{equation*}

A measure of the asymmetry of an open, bounded, convex set in $\R^d$ can be given via the John's ellipsoid. John's theorem \cite{J}, asserts the existence of a unique, open, convex ellipsoid $E(d^{-1}a)$ with semi-axes $d^{-1}a_1,...,d^{-1}a_d$ such that
\begin{equation}\label{e21}
E(d^{-1}a)\subset\Om\subset E(a),
\end{equation}
and that among all ellipsoids in $\Om$, $E(d^{-1}a)$ has maximal measure. See also pp.13--18 in \cite{KB}.

\medspace

We remark that without the convexity constraint in \eqref{e19} we have for $d\ge 3$,
\begin{equation}\label{e19a}
\sup\{G(\Om) :\, \Om\textup{ non-empty, open and bounded}\}=+\infty.
\end{equation}
To see this we let for $k\in \N$,
$$
  \Omega_k = \cup_{j=1}^k B_{r_j}(x_j),
$$
$$
r_j=j^{-\beta},\,j\in\{1,2,...,k\},
$$
and where the $x_j$ are such that $|x_i-x_j|\ge (10k)!,\,i\ne j,\,(i,j)\in\{1,2,...,k\}^2.$
Then, for $\beta\in (d^{-1},(d-2)^{-1}),$
$|\Omega_k|\le \omega_d\zeta(\beta d)<\infty$. By \eqref{e4} and definition \eqref{e8e}, $T(\Omega_k)\ge T(B_1)=\tau_d.$
Since the balls in $\Om_k$ are increasingly well separated as $k$ becomes large, the Newtonian capacity of this finite union is approximately additive, and approximately equal to the sum of the capacities of the balls in $\Om_k$.
Hence by \eqref{e12} and \eqref{e15},
$$
\cp(\Omega_k)\ge \sum_{j=1}^{k}\kappa_d j^{-\beta(d-2)}\big(1-o(1)\big),\quad k\rightarrow\infty.
$$
Combining  the inequalities above we obtain by \eqref{e18},
\begin{align*}
\sup\{G(\Om) :\, &\Om\textup{ non-empty, open and bounded}\}\nonumber\\& \ge G(\Om_k)\nonumber\\&\ge G(B_1)\zeta(\beta d)^{-2}\sum_{j=1}^{k}\kappa_d j^{-\beta(d-2)}\big(1-o(1)\big),
\end{align*}
which tends to $+\infty$ for $k\rightarrow\infty$. This implies \eqref{e19a}.

A similar construction works for the functional $H$, and also for the functionals $G_{\alpha},H_\alpha$ (with appropriate choices of $\beta$) considered in Sections \ref{sec3} and \ref{sec4} respectively.

\section{Measure constraint \label{sec2}}

Theorem \ref{the1} below reduces the geometrical complexity of elements in a maximising sequence for $G$ with a  measure constraint.
\begin{theorem}\label{the1}
Let $\Om$ be an element of a maximising sequence of the variational expression defined in \eqref{e19}, with $G(\Om)\ge G(B_1)$, and with John's ellipsoid $E(d^{-1}a)$.
Let $\pi_{\Om}(a)=(b_1,...,b_d)$ be a permutation of $a\in \R_+^d$ such that $b_1\ge b_2\ge...\ge b_d$. If $d\ge 3$, then
\begin{equation}\label{e22}
\frac{b_{d-2}}{b_d}\le e^{2^{(d-2)/2}d^{2d+1}/(d-2)}.
\end{equation}
\end{theorem}
\begin{proof}
By \eqref{e7}
\begin{equation}\label{e23}
T(\Om)\le T(E(a))\le d\tau_d\bigg(\prod_{i\le d}b_i\bigg)b_d^2,
\end{equation}
and by \eqref{e21}
\begin{equation}\label{e24}
|\Om|\ge |E(d^{-1}a)|=d^{-d}\omega_{d}\prod_{i\le d}b_i.
\end{equation}
Throughout we let $c_i=b_i^{-2},\,i=1,...,d$. By \eqref{e14},
\begin{align}\label{e25}
\mathfrak{e}(a)&=\bigg(\prod_{i\le d}b_i\bigg)^{-1}\int_0^{\infty}dt\, \prod_{i\le d}(1+c_it)^{-1/2}\nonumber \\ &
\ge \bigg(\prod_{i\le d}b_i\bigg)^{-1}\int_0^{\infty}dt\, (1+c_dt)^{-1}\prod_{i\le d-2}(1+c_it)^{-1/2}\nonumber \\ &
\ge \bigg(\prod_{i\le d}b_i\bigg)^{-1}\int_0^{\infty}dt\, (1+c_dt)^{-1}(1+c_{d-2}t)^{(2-d)/2}\nonumber \\ &
=\bigg(\prod_{i\le d}b_i\bigg)^{-1}b_{d-2}^2\int_0^{\infty}dt\, \bigg(1+\frac{c_d}{c_{d-2}}t\bigg)^{-1}(1+t)^{(2-d)/2}\nonumber\\&
\end{align}
\begin{align}
\hspace{17mm}&\ge\bigg(\prod_{i\le d}b_i\bigg)^{-1}b_{d-2}^2\int_0^{1}dt\, \bigg(1+\frac{c_d}{c_{d-2}}t\bigg)^{-1}(1+t)^{(2-d)/2}\nonumber\\&
\ge2^{(2-d)/2}\bigg(\prod_{i\le d}b_i\bigg)^{-1}b_{d-2}^2\int_0^{1}dt\, \bigg(1+\frac{c_d}{c_{d-2}}t\bigg)^{-1}\nonumber \\&
=2^{(2-d)/2}\bigg(\prod_{i\le d}b_i\bigg)^{-1}b_{d}^2\log\bigg(1+\frac{b_{d-2}^2}{b_d^2}\bigg).
\end{align}
By \eqref{e13}, \eqref{e23}, \eqref{e24} and \eqref{e25},
\begin{equation}\label{e26}
G(\Omega)\le \frac{2^{d/2}d^{2d+1}}{d-2}\frac{\tau_d\kappa_d}{\omega_d^2}\bigg(\log\bigg(1+\frac{b_{d-2}^2}{b_d^2}\bigg)\bigg)^{-1}.
\end{equation}
By hypothesis
\begin{equation}\label{e27}
G(\Omega)\ge G(B_1)=\frac{\kappa_d\tau_d}{\omega_d^2}.
\end{equation}
By \eqref{e26}  and \eqref{e27},
\begin{equation*}
1\le \frac{2^{d/2}d^{2d+1}}{d-2}\bigg(\log\bigg(1+\frac{b_{d-2}^2}{b_{d}^2}\bigg)\bigg)^{-1}.
\end{equation*}
This implies
\begin{equation*}
\frac{b_{d-2}^2}{b^2_d}\le e^{2^{d/2}d^{2d+1}/(d-2)},
\end{equation*}
which in turn implies \eqref{e22}.
\end{proof}

\medskip
In the statement and proof of Theorem \ref{the2} below it is convenient to denote the eccentricity for an ellipse $E(a)\subset \R^d$ by
\begin{equation}\label{e30}
C(a)=\frac{1}{d-1}\sum_{i=2}^d\frac{b_1^2}{b_i^2}.
\end{equation}
We note that $C(a)\ge 1$, with equality if and only if $E(a)$ is a ball. Furthermore we have
\begin{equation*}
C(a)\ge\frac{b_1^2}{(d-1)b_d^2}.
\end{equation*}

\begin{theorem}\label{the2}
Let $\mathfrak E_d$ denote the collection of open ellipsoids in $\R^d$.
\begin{enumerate}
\item[\textup{(i)}]If $d\ge 3$, then
\begin{equation}\label{e32}
\sup\{G(\Omega): \Omega\in \mathfrak{E}_d\}=G(B_1),
\end{equation}
and the supremum in the left-hand side of \eqref{e32} is achieved if and only if $\Omega$ is a ball.
\item[\textup{(ii)}]If $d\ge3$, then
\begin{equation}\label{e32a}
\sup\{G(\Om) :\, \Om\textup{ non-empty, open, bounded and convex}\}\le d^{2d}G(B_1).
\end{equation}
\item[\textup{(iii)}]If $d\ge 4$, then
\begin{equation}\label{e33}
G(E(a))\le G(B_1)\frac{d(d-3)}{(d-1)(d-2)}\bigg(1-\frac{1}{1+C(a)^{1/2}}\bigg)^{-1}.
\end{equation}
\end{enumerate}
\end{theorem}
\begin{proof}
(i) By \eqref{e13} and \eqref{e14},
\begin{equation}\label{e34}
\cp(\overline{E(a)})=\frac{\kappa_d\prod_{i\le d}b_i}{\frac{d}{2}-1}\bigg(\int_0^\infty dt\prod_{i\le d}\Big(1+\frac{t}{b_i^2}\Big)^{-1/2}\bigg)^{-1}.
\end{equation}
By \eqref{e7} and \eqref{e34},
\begin{equation}\label{e35}
G(E(a))=\frac{\kappa_d\tau_d}{\omega_d^2}\frac{2d}{d-2}\bigg(\int_0^{\infty}dt\,\frac{c_1+...+c_d}{((1+c_1t)...(1+c_dt))^{1/2}}\bigg)^{-1}.
\end{equation}
By the geometric-arithmetic mean inequality,
\begin{align}\label{e36}
((1+c_1t)...(1+c_dt))^{1/2}&\le \bigg(\frac{1}{d}\sum_{i=1}^d(1+c_it)\bigg)^{d/2}\nonumber \\ &
= \bigg(1+\frac{c_1+...+c_d}{d}t\bigg)^{d/2}.
\end{align}
The change of variable
\begin{equation}\label{e37}
\theta=\frac{c_1+...+c_d}{d}t,
\end{equation}
yields by \eqref{e35}, \eqref{e36} and \eqref{e37},
\begin{align*}
G(E(a))&\le\frac{\kappa_d\tau_d}{\omega_d^2}\frac{2}{d-2}\bigg(\int_0^{\infty}d\theta\ (1+\theta)^{-d/2}\bigg)^{-1}\nonumber \\ &
=G(B_1).
\end{align*}
By scaling invariance $G(B_1)= G(B_r),\,r>0.$ Since $B_r\in \mathfrak{E}_d$, $B_r$ is a maximiser. To prove the if and only if part of the assertion we note that we have equality in \eqref{e39} if and only if all $c_i$'s are equal. That is if and only if $\Omega$ is a ball.

(ii) By \eqref{e21} and monotonicity of $|\Om|,\,T(\Om)$ and $\cp(\overline{\Om}),$ we obtain
\begin{align*}
G(\Om)&\le \frac{T(E(a))\cp(\overline{E(a)})}{|E(d^{-1}a)|^2}\nonumber\\&
=d^{2d}\frac{T(E(a))\cp(\overline{E(a)})}{|E(a)|^2}\nonumber\\&
\le d^{2d}G(B_1),
\end{align*}
where we have used (i) in the final inequality.

(iii) By the geometric-arithmetic mean inequality
\begin{align}\label{e39}
((1+c_2t)...(1+c_dt))^{1/2}&\le \bigg(\frac{1}{d-1}\sum_{i=2}^d(1+c_it)\bigg)^{(d-1)/2}\nonumber \\ &
= \bigg(1+\frac{c_2+...+c_d}{d-1}t\bigg)^{(d-1)/2}.
\end{align}
By \eqref{e35}, \eqref{e39}, and the change of variables $\theta=c_1C(a)t$,
\begin{align}\label{e40}
G(E(a))&\le G(B_1)\frac{2dC(a)}{(d-2)(1+(d-1)C(a))}\nonumber \\ &
\hspace{35mm}\times\bigg(\int_0^{\infty}d\theta\,\big(1+C(a)^{-1}\theta\big)^{-1/2}(1+\theta)^{(1-d)/2}\bigg)^{-1}\nonumber \\ &
\le G(B_1)\frac{2d}{(d-1)(d-2)}\bigg(\int_0^{\infty}d\theta\,\big(1+C(a)^{-1}\theta\big)^{-1/2}(1+\theta)^{(1-d)/2}\bigg)^{-1}.
\end{align}
An integration by parts yields
\begin{align}\label{e41}
\int_0^{\infty}d\theta\,\big(1+C(a)^{-1}\theta\big)^{-1/2}&(1+\theta)^{(1-d)/2}\nonumber \\ &
=\frac{2}{d-3}\bigg(1-\int_0^{\infty}\frac{d\theta}{2C(a)}\,(1+\theta)^{(3-d)/2}\big(1+C(a)^{-1}\theta\big)^{-3/2}\bigg)\nonumber\\&
\ge \frac{2}{d-3}\bigg(1-\int_0^{\infty}\frac{d\theta}{2C(a)}\,(1+\theta)^{-1/2}\big(1+C(a)^{-1}\theta\big)^{-3/2}\bigg)\nonumber \\ &
=\frac{2}{d-3}\bigg(1-\frac12\int_0^{\infty}d\theta\,(1+C(a)\theta)^{-1/2}\big(1+\theta\big)^{-3/2}\bigg)\nonumber \\ &
=\frac{2}{d-3}\bigg(1-\frac{1}{1+C(a)^{1/2}}\bigg),
\end{align}
where the final integral in the right-hand side of \eqref{e44} has been evaluated using the change of variables $1+\theta=\psi^{-2}$. Theorem \ref{the2}(iii) follows by \eqref{e40} and \eqref{e41}.
\end{proof}

It is easily seen, by considering the ellipsoid $E(a)$ with $a=(b_1,1,...,1)$ and by letting $b_1\rightarrow\infty$, that the factor $\frac{d(d-3)}{(d-1)(d-2)}$ in the right-hand side of \eqref{e33} is sharp.

\medskip

Below we shall state and prove the existence of a maximiser for $G$ on a suitable class of open, convex sets with fixed measure $\omega_d$. Let $\varepsilon>0$, and let
\begin{equation}\label{e42}
\mathfrak{E}_d(\varepsilon)=\{\Omega\subset\R^d\, \textup{open, convex}: |\Omega|=\omega_d, \big(\exists E\in \mathfrak{E}_d,\, |E|\le \omega_d(1+\varepsilon),\,\Omega\subset E\big) \},
\end{equation}
be the collection of all open, convex sets in $\R^d$ with fixed measure $\omega_d$ which are contained in an ellipsoid of measure at most $\omega_d(1+\varepsilon)$.
As we have the existence of a convex maximiser for the full variational problem \eqref{e19} if $d=3,$ we subsequently only consider the case $d\ge 4$.
Our main existence result, stated below, is for small $\varepsilon$.
\begin{theorem}\label{the3}
If $d\ge 4$, and if
\begin{equation}\label{e43}
\varepsilon< \bigg(\frac{(d-1)(d-2)}{d(d-3)}\bigg)^{1/2}-1,
\end{equation}
then the variational problem
\begin{equation}\label{e44}
\mathfrak{g}_d(\varepsilon):=\sup\{G(\Om):\Om\in\mathfrak{E}_d(\varepsilon)\}
\end{equation}
has an open, convex maximiser $\Omega^{\varepsilon}$ with measure $\omega_d$,
and with
\begin{equation}\label{e45}
\frac{\textup{diam}(\Omega^{\varepsilon})}{r(\Omega^{\varepsilon})}\le 2^{d}\bigg(\frac{d(d-1)^d(d-2)}{d-3}\bigg)^{1/2}\bigg(1-\frac{d(d-3)}{(d-1)(d-2)}(1+\varepsilon)^2\bigg)^{1-d}.
\end{equation}
\end{theorem}

We see that the upper bound in \eqref{e45} diverges as $\varepsilon$ increases to the critical value in the right-hand side of \eqref{e43}.
\begin{proof}
By \eqref{e32a},
\begin{equation*}
\mathfrak{g}_d(\varepsilon)\le d^{2d}G(B_1).
\end{equation*}
If $\mathfrak{g}_d(\varepsilon)=G(B_1)$, then $B_1\in \mathfrak{E}_d(\varepsilon)$ is a maximiser which satisfies \eqref{e45}, and there is nothing to prove.
If $\mathfrak{g}_d(\varepsilon)>G(B_1)$, then
let $(\Omega_n)$ be a maximising sequence with $\Omega_n\in \mathfrak{E}_d(\varepsilon),\,n\in \N$. We may assume that $G(\Omega_n)>G(B_1),\, n\in \N$, and that $(G(\Omega_n))$ is increasing. Let $\Omega_n\in\mathfrak{E}_d(\varepsilon),\, \Omega_n\subset E(a^{(n)})$, with $|E(a^{(n)})|\le \omega_d(1+\varepsilon)$.
By monotonicity of both torsion and Newtonian capacity,
\begin{align}\label{e47}
G(\Omega_n)&= \frac{\cp(\overline{\Omega_n})T(\Omega_n)}{|\Omega_n|^2}\nonumber \\ &
\le \frac{\cp(\overline{E(a^{(n)})})T(E(a^{(n)}))}{|\Omega_n|^2}\nonumber\\ &
\le \frac{\cp(\overline{E(a^{(n)})})T(E(a^{(n)})}{|E(a^{(n)})|^2}(1+\varepsilon)^2\nonumber \\ &
=G(E(a^{(n)}))(1+\varepsilon)^2.
\end{align}
We denote the eccentricity for $E(a^{(n)}),$ defined in \eqref{e30}, by $C(a^{(n)})$.

We have the following dichotomy: (i)  $\lim_{n\rightarrow\infty} C(a^{(n)})=+\infty$ or $C(a^{(n)})$ has a convergent subsequence, again denoted by $C(a^{(n)})$, with $\lim_{n\rightarrow\infty}C(a^{(n)})=c$, where $c\in [1,\infty)$.
This in turn implies that, for a suitable subsequence, each of the quotients defining $C(a^{(n)})$ converges. We write $\lim_{n\rightarrow\infty}C(a^{(n)})=C(a)=\frac{1}{d-1}\sum_{i=2}^d\lim_{n\rightarrow\infty} \big(b_1^{(n)}/b_i^{(n)}\big)^2$. We now consider the two cases of the dichotomy.

(i) By Theorem \ref{the2},
\begin{align*}
\lim_{n\rightarrow\infty}G(\Omega_n)&\le G(B_1)\frac{d(d-3)}{(d-1)(d-2)}(1+\varepsilon)^2\limsup_{n\rightarrow\infty}\bigg(1-\frac{1}{1+C(a^{(n)})^{1/2}}\bigg)^{-1}\nonumber \\ &
=G(B_1)\frac{d(d-3)}{(d-1)(d-2)}(1+\varepsilon)^2\nonumber \\ &
<G(B_1),
\end{align*}
by the hypothesis \eqref{e43} on $\varepsilon$. This contradicts the assumption that $G(\Omega_n)>G(B_1),\, n\in \N$, and that $(G(\Omega_n))$ is increasing.

(ii) By Theorem \ref{the2}, \eqref{e47} and the hypothesis on $C(a^{(n)})$,
\begin{align}\label{e49}
\lim_{n\rightarrow\infty}G(\Omega_n)\le G(B_1)\frac{d(d-3)}{(d-1)(d-2)}(1+\varepsilon)^2\bigg(1-\frac{1}{1+C(a)^{1/2}}\bigg)^{-1}.
\end{align}
The left-hand side of \eqref{e49} is strictly greater than $G(B_1)$. This gives
\begin{equation*}
C(a)\le \bigg(1-\frac{d(d-3)}{(d-1)(d-2)}(1+\varepsilon)^2\bigg)^{-2}.
\end{equation*}
We have, using \eqref{e30},
\begin{equation}\label{e51}
\lim_{n\rightarrow\infty}\frac{b_1^{(n)}}{b_i^{(n)}}\le (d-1)^{1/2}\bigg(1-\frac{d(d-3)}{(d-1)(d-2)}(1+\varepsilon)^2\bigg)^{-1},\quad i=2,...,d.
\end{equation}
Denote the constant in the right-hand side of \eqref{e51} by $\mathfrak{c}$. Since
\begin{equation}\label{e52}
1+\varepsilon\ge \prod_{i=1}^db^{(n)}_i=(b^{(n)}_1)^d\prod_{i=1}^d\frac{b^{(n)}_i}{b^{(n)}_1},\,n\in \N,
\end{equation}
we have by \eqref{e51} and \eqref{e52}
\begin{equation*}
1+\varepsilon\ge \mathfrak{c}^{1-d}\limsup_{n\rightarrow\infty}(b^{(n)}_1)^d.
\end{equation*}
By extracting a further subsequence and relabeling $\lim_{n\rightarrow\infty}b_1^{(n)}=b_1$. Since all quotients converge, $\lim_{n\rightarrow\infty}b_i^{(n)}=b_i,\, i=1,...,d.$ Hence
\begin{align}\label{e54}
\limsup_{n\rightarrow\infty}\textup{diam}(\overline{\Omega_n})&=\limsup_{n\rightarrow\infty}\textup{diam}(\Omega_n)\nonumber \\
&\le 2b_1\nonumber \\ &\le 2\mathfrak{c}^{(d-1)/d}(1+\varepsilon)^{1/d}\nonumber \\ &
\le 2\bigg(\frac{(d-1)(d-2)}{d(d-3)}\bigg)^{1/(2d)}\mathfrak{c}^{(d-1)/d}.
\end{align}
The right-hand side of \eqref{e54} depends on $d$ and on $\varepsilon$ only.
So there exists a subsequence of possible translates of $(\overline{\Omega_n})$ which converges both in the Hausdorff metric and the complementary Hausdorff metric to some compact convex set $\overline{\Omega^{\varepsilon}}$.
Since measure and diameter are continuous  $|\overline{\Omega^{\varepsilon}}|=|\Omega^{\varepsilon}|=\omega_d$,
and
\begin{equation}\label{e55}
\textup{diam}(\Omega^{\varepsilon})\le 2\bigg(\frac{(d-1)(d-2)}{d(d-3)}\bigg)^{1/(2d)}\mathfrak{c}^{(d-1)/d}.
\end{equation}
Moreover since both capacity and torsional rigidity are continuous set functions in these metrics on the class of bounded convex sets (see for example Chapter 4 in \cite{BB}) we have that $\Omega^{\varepsilon}$ is a maximiser possibly dependent upon $\varepsilon$.

For an open set $\Omega\subset \R^d$ we denote its perimeter by $P(\Omega)$. By Proposition 2.4.3 (iii) and (i) in \cite{BB} we have for an open, bounded and convex $\Omega,$
\begin{align}\label{e57}
|\Omega|&\le r(\Omega)P(\Omega)\nonumber\\&\le r(\Om)P(B_{\textup{diam}(\Omega)})\nonumber\\&=
 d\omega_dr(\Omega)\textup{diam}(\Omega)^{d-1},
\end{align}
where we have used that $\Om$ is contained in a ball of radius $\textup{diam}(\Omega)$. Applying \eqref{e57} to the convex set $\Omega^{\varepsilon}$ with $|\Omega^{\varepsilon}|=\omega_d$ yields
\begin{equation}\label{e58}
\frac{\textup{diam}(\Omega^{\varepsilon})}{r(\Omega^{\varepsilon})}\le d\,\textup{diam}(\Omega^{\varepsilon})^d.
\end{equation}
This implies \eqref{e45} by \eqref{e55}, \eqref{e57} and the definition of $\mathfrak{c}$.
\end{proof}

The right-hand side of \eqref{e45} diverges as $\varepsilon$ increases to the right-hand side of \eqref{e43}.
Note that $\varepsilon\mapsto G(\Omega^{\varepsilon})$ is non-decreasing. It would be interesting to show that the map is constant on $[0,\delta)$ for some possibly small $\delta>0$. This, together with Theorem \ref{the2}(i), would support the conjecture that $B_1\subset \R^d$ is a maximiser of the right-hand side of \eqref{e18} on the collection of open, bounded, convex sets in $\R^d, d\ge 3$.

We obtain very similar results if we change the $\omega_d$'s in the right-hand side of \eqref{e42} by some fixed constant $v,v>0$. The convex maximiser $\Omega^{\varepsilon,v}$ then has measure $v$, and $\frac{\textup{diam}(\Omega^{\varepsilon,v})}{r(\Omega^{\varepsilon,v})}$ is bounded from above by the right-hand side of \eqref{e45}.

\section{Logarithmic capacity \label{sec3}}

In this section we denote by $\cp(\cdot)$ the logarithmic capacity, defined on the class of compact sets in $\R^2$, and recall its definition below.
Let $\mu$ be a probability measure supported on $K$, and let
\begin{equation*}
I(\mu)=\iint_{K\times K}\log\Big(\frac{1}{|x-y|}\Big)\mu(dx)\mu(dy).
\end{equation*}
Furthermore let
\begin{equation*}
V(K)=\inf\big\{I(\mu):\mu \textup{ a probability measure on $K$} \big\}.
\end{equation*}
The logarithmic capacity of $K$ is denoted by $\cp(K)$, and is the non-negative real number
\begin{equation*}
\cp(K)=e^{-V(K)}.
\end{equation*}

The logarithmic capacity is an increasing set function, and satisfies \eqref{e11} for compact sets $K_1$ and $K_2$.
For an ellipsoid with semi-axes $a_1$ and $a_2$,
\begin{equation}\label{e63}
\cp(\overline{E(a)})=\frac12(a_1+a_2).
\end{equation}
See \cite{L}.

The results for $H$, defined in \eqref{e64}, below are the planar versions of Theorems \ref{the2} and \ref{the3} respectively.
\begin{theorem}\label{the4}
Let $\mathfrak E_2$ denote the collection of open ellipses in $\R^2$. Then
\begin{enumerate}
\item[\textup{(i)}]
\begin{equation}\label{e65}
\sup\{H(\Omega): \Omega\in \mathfrak{E}_2\}=H(B_1),
\end{equation}
and the supremum in the left-hand side of \eqref{e65} is achieved if and only if $\Omega$ is a ball.
\item[\textup{(ii)}]
\begin{equation}\label{e65a}
\sup\{H(\Om) :\, \Om\textup{ non-empty, open, bounded, planar and convex}\}\le 8H(B_1).
\end{equation}
\item[\textup{(iii)}]If $b_1\ge b_2$, then
\begin{equation*}
H(E(b_1,b_2))\le 2^{-1/2}H(B_1)\bigg(1+\frac{b_2}{b_1}\bigg).
\end{equation*}
\end{enumerate}
\end{theorem}
\begin{proof}
(i) By \eqref{e7}, \eqref{e63} and $\Omega=E(b_1,b_2)$,
\begin{equation}\label{e67}
H(E(b_1,b_2))=\frac{b_1+b_2}{4\pi(b_1^2+b_2^2)^{1/2}}
\le 2^{-3/2}\pi^{-1}=H(B_1).
\end{equation}
By scaling invariance $H(B_1)= H(B_r),\,r>0.$ Since $B_r\in \mathfrak{E}_2$, $B_r$ is a maximiser. To prove the if and only if part of the assertion we note that we have equality in \eqref{e67} if and only if $b_1=b_2$. That is if and only if $E$ is a ball.

(ii) By \eqref{e65} and monotonicity of $|\Om|,\,T(\Om)$ and $\cp(\overline{\Om}),$ we obtain
\begin{align*}
H(\Om)&\le \frac{T(E(a))^{1/2}\cp(\overline{E(a)})}{|E(2^{-1}a)|^{3/2}}\nonumber\\&
=8\frac{T(E(a))^{1/2}\cp(\overline{E(a)})}{|E(a)|^{3/2}}\nonumber\\&
\le 8H(B_1),
\end{align*}
where we have used (i) in the final inequality.

(iii) By \eqref{e67}
\begin{align*}
H(E(b_1,b_2))&=\frac{1+\frac{b_2}{b_1}}{4\pi\bigg(1+\frac{b_2^2}{b_1^2}\bigg)^{1/2}}\nonumber \\ &
\le (4\pi)^{-1}\bigg(1+\frac{b_2}{b_1}\bigg)\nonumber \\&
=2^{-1/2}H(B_1)\bigg(1+\frac{b_2}{b_1}\bigg).
\end{align*}
\end{proof}

\begin{theorem}\label{the5}
If
\begin{equation}\label{e69}
\varepsilon< 2^{1/3}-1
\end{equation}
then the variational problem
\begin{equation}\label{e70}
\mathfrak{h}(\varepsilon):=\sup\{H(E):E\in\mathfrak{E}_2(\varepsilon)\}
\end{equation}
has an open, convex maximiser $\Omega^{\varepsilon}$ with measure $\omega_2$,
and with
\begin{equation}\label{e71}
\frac{\textup{diam}(\Omega^{\varepsilon})}{r(\Omega^{\varepsilon})}\le \frac{2^{11/3}}{2^{1/3}-1-\varepsilon}.
\end{equation}
\end{theorem}

The strategy of the proof is along similar lines to the proof of Theorem \ref{the3}. The computations are as follows.
\begin{proof}
By \eqref{e65a},
\begin{equation*}
\mathfrak{h}(\varepsilon)\le 8H(B_1).
\end{equation*}
By domain monotonicity of both $T$ and $\cp$ on the class of convex sets, we have for $\Omega\in \mathfrak{E}_2(\varepsilon)$ by Theorem \ref{the4},
\begin{align*}
H(\Omega)&\le\frac{T(E(a))^{1/2}\cp(\overline{E(a)})}{|\Omega|^{3/2}}\nonumber \\ &
\le\frac{T(E(a))^{1/2}\cp(\overline{E(a)})}{|E(a)|^{3/2}}(1+\varepsilon)^{3/2}\nonumber \\ &
=2^{-1/2}H(B_1)\bigg(1+\frac{b_2}{b_1}\bigg)(1+\varepsilon)^{3/2}.
\end{align*}
If $\mathfrak{h}(\varepsilon)=H(B_1)$, then $B_1\in \mathfrak{E}_2(\varepsilon)$ is a maximiser of \eqref{e70} which satisfies \eqref{e71}, and there is nothing to prove.
If $\mathfrak{h}(\varepsilon)>H(B_1)$, then let $(\Omega_n)$ be a maximising sequence with $\Omega_n\in \mathfrak{E}(\varepsilon),\,n\in \N$. We may assume that $H(\Omega_n)>H(B_1),\, n\in \N$, and that $(H(\Omega_n))$ is increasing. Let $E(a^{(n)})\in\mathfrak{E}_2(\varepsilon),\, \Omega_n\subset E(a^{(n)})$, and let
\begin{equation*}
\mathfrak{b}^{(n)}=\frac{b_2^{(n)}}{b_1^{(n)}},\quad n\in \N.
\end{equation*}
We have the following dichotomy. (i) $\liminf_{n\rightarrow\infty}\mathfrak{b}^{(n)}=0$. (ii) $\liminf_{n\rightarrow\infty}\mathfrak{b}^{(n)}=\mathfrak{b}>0$.

(i) By choosing a further subsequence we may assume that $\lim_{n\rightarrow\infty}\mathfrak{b}^{(n)}=0$. By \eqref{e71} and \eqref{e69},
\begin{align*}
H(B_1)&<\lim_{n\rightarrow\infty}H(\Omega_n)\nonumber\\&\le \lim_{n\rightarrow\infty}2^{-1/2}H(B_1)\big(1+\mathfrak{b}^{(n)}\big)(1+\varepsilon)^{3/2}\nonumber\\&
=2^{-1/2}H(B_1)(1+\varepsilon)^{3/2}\nonumber\\&<H(B_1),
\end{align*}
which is impossible. 

(ii) By choosing a further subsequence we may assume that $\lim_{n\rightarrow\infty}\mathfrak{b}^{(n)}=\mathfrak{b}>0$.
By Theorem \ref{the4}(ii)
\begin{align*}
H(B_1)&<\lim_{n\rightarrow\infty}H(\Omega_n)\nonumber\\&
\le 2^{-1/2}H(B_1)\big(1+\mathfrak{b}\big)(1+\varepsilon)^{3/2}.
\end{align*}
This gives
\begin{equation}\label{e76}
\mathfrak{b}>2^{1/2}(1+\varepsilon)^{-3/2}-1.
\end{equation}
We obtain by \eqref{e76},
\begin{align}\label{e77}
1+\varepsilon&\ge \limsup_{n\rightarrow\infty} b_1^{(n)}b_2^{(n)}\nonumber\\&
\ge \Big(2^{1/2}(1+\varepsilon)^{-3/2}-1\Big)\limsup_{n\rightarrow\infty} (b_1^{(n)})^2.
\end{align}
Hence the sequence $(b_1^{(n)})$ is bounded from above, and there exists a convergent subsequence, again denoted by $(b_1^{(n)}),$ with $\lim_{n\rightarrow\infty}b_1^{(n)}=b_1$.
Existence of a maximiser follows the lines below \eqref{e54}.
Since $\lim_{n\rightarrow\infty}\mathfrak{b}^{(n)}=\mathfrak{b}>0$ we have $b_2=\mathfrak{b}b_1$. By \eqref{e77}, the first two lines in \eqref{e54} and \eqref{e58}, we have
\begin{equation*}
\frac{\textup{diam}(\Omega^{\varepsilon})}{r(\Omega^{\varepsilon})}\le 8b_1^2
\le \frac{8(1+\varepsilon)}{2^{1/2}(1+\varepsilon)^{-3/2}-1},
\end{equation*}
which implies \eqref{e71} by \eqref{e69}.
\end{proof}
The right-hand side of \eqref{e71} diverges as $\varepsilon$ increases to the right-hand side of \eqref{e69}.

\section{Perimeter and measure constraints \label{sec4}}

In this section we investigate the maximisation of
\begin{equation}\label{e79}
G_{\alpha}(\Om)=\frac{T(\Om)\cp(\overline{\Om})}{|\Om|^{\alpha}P(\Om)^{d(2-\alpha)/(d-1)}},
\end{equation}
where $d\ge 3$ and $0\le \alpha\le 2$.
Recall that $P$ scales as
\begin{equation}\label{e80}
P(t\Om)=t^{d-1}P(\Om),\, t>0.
\end{equation}
By \eqref{e5}, \eqref{e12} and \eqref{e80}, we see that $G_{\alpha}$ is scaling invariant. The functional interpolates between the two cases $\alpha=2$ (which was investigated in Section \ref{sec2}), and $\alpha=0$.
We see from the results below that the presence of a perimeter term in  $G_{\alpha}$ guarantees the existence of a maximiser.
\begin{theorem}\label{the6}
\begin{enumerate}
\item[\textup{(i)}]Let $\mathfrak E_d$ denote the collection of open ellipsoids in $\R^d$. If $d\ge 3$ and $0\le \alpha\le2$, then
\begin{equation}\label{e89}
\sup\{G_{\alpha}(\Omega): \Omega\in \mathfrak{E}_d\}=G_{\alpha}(B_1),
\end{equation}
and the supremum in the left-hand side of \eqref{e89} is achieved if and only if $\Omega$ is a ball.
\item[\textup{(ii)}] If $d\ge 3$ and $0\le \alpha\le2$, then
\begin{equation}\label{e81}
 \sup\{G_{\alpha}(\Om):\Om\,\, \textup{non-empty, open, bounded and convex}\}  \le d^{2d}G_{\alpha}(B_1).
\end{equation}
\item[\textup{(iii)}]
 If $0\le\alpha<2$, then the  variational problem in the left-hand side of \eqref{e81} has a maximiser. If $\Om_{\alpha}$ is such a maximiser, then
	\begin{equation}\label{e82}
		\frac{\textup{diam}(\Om_{\alpha})}{r(\Om_{\alpha})}\le 2d^{(2d^2+2d-2d\alpha+2-\alpha)/(2-\alpha)}.
	\end{equation}
\end{enumerate}
\end{theorem}
\begin{proof}
(i) By the isoperimetric inequality
\begin{equation}\label{e90}
P(\Om)\ge \Big(\frac{|\Om|}{|B_1|}\Big)^{(d-1)/d}P(B_1),
\end{equation}
\eqref{e79}, and \eqref{e32},
\begin{align}\label{e91}
\sup\{G_{\alpha}(\Omega): \Omega\in \mathfrak{E}_d\}&\le \bigg(\frac{|B_1|^{(d-1)/d}}{P(B_1)}\bigg)^{d(2-\alpha)/(d-1)}\sup\{G(\Omega): \Omega\in \mathfrak{E}_d\}\nonumber \\ &
=\bigg(\frac{|B_1|^{(d-1)/d}}{P(B_1)}\bigg)^{d(2-\alpha)/(d-1)}G(B_1)\nonumber \\ &=\bigg(\frac{|B_1|^{(d-1)/d}}{P(B_1)}\bigg)^{d(2-\alpha)/(d-1)}\frac{\kappa_d\tau_d}{|B_1|^2}\nonumber \\ &
=G_{\alpha}(B_1).
\end{align}
By scaling invariance $G_{\alpha}(B_1)= G_{\alpha}(B_r),\,r>0.$ Since $B_r\in \mathfrak{E}_d$, $B_r$ is a maximiser. To prove the if and only if part of the assertion we note that by the assertion of Theorem \ref{the2}(i) we have equality in \eqref{e91} if and only if we have equality both in \eqref{e90}. That is if and only if $\Om$ is a ball.

(ii) By \eqref{e32a} and \eqref{e90},
\begin{align}\label{e83}
G_{\alpha}(\Omega)&=G(\Omega)|\Omega|^{2-\alpha}P(\Omega)^{d(\alpha-2)/(d-1)}\nonumber \\&
\le d^{2d}G(B_1)|\Omega|^{2-\alpha}P(\Omega)^{d(\alpha-2)/(d-1)}\nonumber\\&
\le d^{2d}G(B_1)|B_1|^{2-\alpha}P(B_1)^{d(\alpha-2)/(d-1)}\nonumber\\&
=d^{2d}G_{\alpha}(B_1).
\end{align}
This implies the assertion.

(iii) To prove the existence of a maximiser, we observe that if the left-hand side of \eqref{e81} equals $G_{\alpha}(B_1)$, then $B_1$ is a maximiser which satisfies \eqref{e82}.
If the left-hand side of \eqref{e81} is greater than $G_{\alpha}(B_1)$, we let $\Om$ be non-empty, bounded, open, and convex, and such that
\begin{equation}\label{e86}
G_{\alpha}(\Om)\ge G_{\alpha}(B_1).
\end{equation} By the first inequality in \eqref{e57} and \eqref{e83},
\begin{equation}\label{e84}
G_{\alpha}(\Om)\le d^{2d}G(B_1)\bigg(\frac{r(\Om)^d}{|\Om|}\bigg)^{(2-\alpha)/(d-1)}.
\end{equation}
By \eqref{e84} and \eqref{e86},
\begin{align}\label{e85}
\bigg(\frac{|\Om|}{r(\Om)^d}\bigg)^{(2-\alpha)/(d-1)}&\le d^{2d}\frac{G(B_1)}{G_{\alpha}(B_1)}\nonumber\\&
=d^{2d}\big(d^d\omega_d\big)^{(2-\alpha)/(d-1)}.
\end{align}
Since $E(d^{-1}a)\subset\Om\subset E(a)$, we have
\begin{equation}\label{e87}
|\Om|\ge d^{-d}|E(a)|=\omega_dd^{-d}\prod_{i\le d}b_i\ge \omega_dd^{-d}b_d^{d-1}b_1,
\end{equation}
and
\begin{equation}\label{e88}
r(\Om)\le b_d.
\end{equation}
By \eqref{e85}, \eqref{e87} and \eqref{e88}, we find
\begin{equation*}
\frac{b_1}{b_d}\le d^{2d(d+1-\alpha)/(2-\alpha)}.
\end{equation*}
Since $\diam (\Om)\le 2b_1$ and $r(\Om)\ge d^{-1}b_d$, we arrive at
\begin{equation}\label{e88b}
\frac{\textup{diam}(\Om)}{r(\Om)}\le 2d^{(2d^2+2d-2d\alpha+2-\alpha)/(2-\alpha)}.
\end{equation}
Hence an element $\Om$ of a  maximising sequence satisfies \eqref{e88}, and so does $\overline{\Om}$. Without loss of generality we may fix $r(\overline{\Om})=1$. Hence translates of elements of a maximising sequence with fixed inradius $1$ are contained in a large closed ball with a radius dependent on $d$ and on $\alpha$.  We now have the standard compactness result of the Hausdorff metric on the compact sets to arrive at the conclusion of the third part of Theorem \ref{the6}.
\end{proof}

Let $d=2,\,0\le \alpha\le \frac32 $, and let
\begin{equation}\label{e92}
H_{\alpha}(\Omega)=\frac{T(\Omega)^{1/2}\cp(\overline{\Omega})}{|\Omega|^{\alpha}P(\Om)^{3-2\alpha}},
\end{equation}
where $\cp$ denotes logarithmic capacity.  We investigate the maximisation problem of \eqref{e92}. In Theorem \ref{the5} we saw that a maximiser exists for $\alpha=\frac32,$ and the constraint $\Om\in \mathfrak{E}_2(\varepsilon)$.
Theorem \ref{the8} below interpolates between $\alpha=\frac32$ and $\alpha=0$.

\begin{theorem}\label{the8}Let $d=2$.
\begin{enumerate}
\item[\textup{(i)}]Let $\mathfrak E_2$ denote the collection of open ellipses in $\R^2$. If $0\le \alpha\le2$, then
\begin{equation}\label{e92e}
\sup\{H_{\alpha}(\Omega): \Omega\in \mathfrak{E}_2\}=H_{\alpha}(B_1),
\end{equation}
and the supremum in the left-hand side of \eqref{e92e} is achieved if and only if $\Omega$ is a ball.
\item[\textup{(ii)}]If $0\le\alpha\le\frac32$, then
\begin{equation}\label{e93}
 \sup\{H_{\alpha}(\Om):\Om\, \textup{non-empty, open, bounded, planar and convex}\}\le 2^{2\alpha}\pi^{3-2\alpha}H_{\alpha}(B_1).
\end{equation}
\item[\textup{(iii)}] If $0\le\alpha<\frac32$, then the variational problem in the left-hand side of \eqref{e93} has a maximiser.
If $\Omega_{\alpha}$ is any such maximiser, then
\begin{equation}\label{e95}
\frac{\textup{diam}(\Omega_{\alpha})}{r(\Omega_{\alpha})}\le 2^{(3+2\alpha)/(3-2\alpha)}\pi^2.
\end{equation}
\item[\textup{(iv)}]If $\alpha=0$, then the variational problem
\begin{equation*}
\sup\big\{H_{0}(\Omega):\Omega\,\textup{\,open, planar, connected, $0<|\Om|<\infty$}\big \},
\end{equation*}
has a maximiser. Any such maximiser is also a maximiser of \eqref{e93} for $\alpha=0$, and henceforth satisfies \eqref{e95}.
\end{enumerate}
\end{theorem}
\begin{proof}
(i) Following \eqref{e91} we have, by \eqref{e65}, \eqref{e90} and \eqref{e92},
\begin{align*}
\sup\{H_{\alpha}(\Omega): \Omega\in \mathfrak{E}_2\}&\le \bigg(\frac{|B_1|^{1/2}}{P(B_1)}\bigg)^{3-2\alpha}\sup\{H(\Omega): \Omega\in \mathfrak{E}_2\}\nonumber \\ &
=\bigg(\frac{|B_1|^{1/2}}{P(B_1)}\bigg)^{3-2\alpha}H(B_1)\nonumber \\ &
=H_{\alpha}(B_1).
\end{align*}

(ii) By the isoperimetric inequality \eqref{e90} with $d=2$, John's theorem with $E(a/2)\subset \Om\subset E(a)$, $b_1\ge b_2$, and \eqref{e63}, we have $\cp(\overline{\Om})\le \cp(\overline{E(a)})\le \frac12(b_1+b_2)$. Furthermore $|\Om|\ge |E(2^{-1}a)|\ge \pi b_1b_2/4$. Since the ellipse $E(2^{-1}a)$ contains a rhombus with axes of length $b_1$ and $b_2$ respectively, we have
\begin{equation}\label{a1}
P(\Om)\ge P(E(2^{-1}a))\ge 2(b_1^2+b_2^2)^{1/2}.
\end{equation}
So by \eqref{e7} for $d=2$,
\begin{equation}\label{e97}
T(\Om)^{1/2}\le \bigg(\frac{\pi b_1^3b_2^3}{4(b_1^2+b_2^2)}\bigg)^{1/2}\le \bigg(\frac{\pi(b_1b_2)^3}{2(b_1+b_2)^2}\bigg)^{1/2}.
\end{equation}
By \eqref{e92}, \eqref{e97} and the inequalities preceding \eqref{e97},
\begin{align*}
H_{\alpha}(\Omega)&\le 2^{(8\alpha-9)/2}\pi^{(1-2\alpha)/2}\bigg(\frac{b_2}{b_1}\bigg)^{(3-2\alpha)/2}.
\end{align*}
This, together with the fact that
\begin{equation}\label{e99}
H_{\alpha}(B_1)=2^{(4\alpha-9)/2}\pi^{(2\alpha-5)/2},
\end{equation}
yields
\begin{equation*}
H_{\alpha}(\Omega)\le 2^{2\alpha}\pi^{3-2\alpha}H_{\alpha}(B_1)\bigg(\frac{b_2}{b_1}\bigg)^{(3-2\alpha)/2}.
\end{equation*}
This implies \eqref{e93}.

(iii) If the supremum in \eqref{e93} equals $H_{\alpha}(B_1)$, then $B_1$ is a maximiser, which satisfies \eqref{e95}. Otherwise let $(\Om_n)$ be a maximising sequence with $(H_{\alpha}(\Om_n))$ increasing and
$H_{\alpha}(\Om_n)>H_{\alpha}(B_1)$. If $\Omega$ is an element of that maximising sequence, then its John's ellipsoid satisfies
\begin{equation}\label{e101}
\frac{b_1}{b_2}\le2^{4\alpha/(3-2\alpha)}\pi^2.
\end{equation}
Hence by \eqref{e101}
\begin{equation*}
\frac{\textup{diam}(\Omega)}{r(\Omega)}\le\frac{\textup{diam}(E(a))}{r(E(2^{-1}a))}
\le \frac{2b_1}{b_2}\le2\cdot2^{4\alpha/(3-2\alpha)}\pi^2,
\end{equation*}
which gives \eqref{e95}.
Existence of a convex maximiser follows the same lines as for example those of Theorems \ref{the3} and \ref{the6}.

(iv) Since convex sets are connected,
\begin{align*}
\sup\big\{H_{0}(\Omega):\,&\Omega\,\textup{\,open, convex, $0<|\Om|<\infty$}\big \}\nonumber\\&
\le \sup\big\{H_{0}(\Omega):\Omega\,\textup{open, planar, connected, $0<|\Om|<\infty$}\big \}.
\end{align*}
To prove the converse we let $\Om$ be open, connected and bounded in $\R^2$ with $0<P(\Om)<\infty$.  Let $\Om^*$ be the interior of the convex envelope of $\Om$. Then
\begin{equation*}
T(\Om)\le T(\Om^*),\, \cp(\Om)\le \cp(\Om^*),\, P(\Om)\ge P(\Om^*),
\end{equation*}
and $\Om^*$ is open, bounded and convex. Hence
\begin{equation}\label{e105}
H_0(\Om)\le H_0(\Om^*)\le \sup\big\{H_{0}(\Omega^*):\Omega\textup{\,open, convex in $\R^2$ with $|\Om^*|<\infty$}\big \}.
\end{equation}
Taking the supremum over all $\Om$ in the left-hand side of \eqref{e105} which are open, connected and bounded in $\R^2$ gives the required inequality. If $\Om_0$ is a maximiser of $H_0$, then it is open, bounded and connected. The assertion follows since both suprema of both variational expressions are equal.
\end{proof}

\section*{Acknowledgements}
MvdB acknowledges support by the Leverhulme Trust through Emeritus Fellowship EM-2018-011-9. AM has been supported by the project {\em Geometric problems with loss of compactness} from the Scuola Normale Superiore and by GNAMPA as part of INdAM. MvdB wishes to thank the Scuola Normale Superiore for its hospitality in January 2020. Both authors are grateful to the referees for their helpful comments.


\begin{thebibliography} {99}

\bibitem{vdBB} M. van den Berg, G. Buttazzo, On capacity and torsional rigidity, Bulletin of the London Mathematical Society \textbf{53} (2021), 347--359.

\bibitem{MvdBBP} M. van den Berg, G. Buttazzo, A. Pratelli,  On relations between principal eigenvalue and torsional rigidity, Communications in Contemporary Mathematics \textbf{23} (2021) 2050093.

\bibitem{MvdBFNT} M. van den Berg, V. Ferone, C. Nitsch, C. Trombetti,
On P\'olya's inequality for torsional rigidity and first Dirichlet eigenvalue, Integral Equations and Operator Theory \textbf{86} (2016), 579--600.

\bibitem{KB} K. M. Ball, An elementary introduction to modern convex geometry, Flavors of geometry 1--58, MSRI Publications 31, Cambridge University Press, Cambridge, 1997.

\bibitem{BDP} L.  Brasco, G. De Philippis,  Spectral inequalities in quantitative form. Shape optimization and spectral theory, 201--281, De Gruyter Open, Warsaw, 2017.

\bibitem{BBP1}L. Briani, G. Buttazzo, F. Prinari, Some inequalities involving perimeter and torsional rigidity, Appl. Math. Optim. \textbf{84} (2021), 2727--2741.

\bibitem{BBP2}L. Briani, G. Buttazzo, F. Prinari, A shape optimization problem on planar sets with prescribed topology, J. Optim. Theory Appl. https://doi.org/10.1007/s10957--021--01870--7.

\bibitem{BB}D. Bucur, G. Buttazzo, Variational methods in shape optimization problems. Progress in Nonlinear Differential Equations and their Applications, 65. Birkh\"auser Boston, Inc., Boston, MA, 2005.

\bibitem{GC}G. Chrystal, Electricity. Encyclopedia Britannica, 9'th ed. (1879).

\bibitem{DPLMV} G. De Philippis, J. Lamboley, M. Pierre, B. Velichkov,  Regularity of minimizers of shape optimization problems involving perimeter, J. Math. Pures Appl. \textbf{109} (2018), 147--181.

\bibitem{FGP} I. Fragal\`{a}, F. Gazzola, M. Pierre,  On an isoperimetric inequality for capacity conjectured by P\'{o}lya and Szeg\"{o}. J. Differential Equations \textbf{250} (2011), 1500--1520.

\bibitem{FL}I. Ftouhi, J. Lamboley, Blaschke-Santal\'o diagram for volume, perimeter and first Dirichlet eigenvalue, SIAM J. Math. Anal. \textbf{53} (2021), 1670--1710.

\bibitem{IMcK}K. It\^o, H. P. McKean, Diffusion processes and their sample paths. Second printing, corrected. Die Grundlehren der mathematischen Wissenschaften, Band 125. Springer-Verlag, Berlin-New York, 1974.

\bibitem{J}
F. John, Extremum problems with inequalities as subsidiary conditions. Studies and Essays Presented to R. Courant on his 60'th Birthday, January 8, 1948. Interscience Publishers, Inc., New York, N. Y., 1948, pp. 187--204.

\bibitem{L}N. S. Landkof, Foundations of Modern Potential Theory, Die Grundlehren der mathematischen Wissenschaften 180 (Springer-Verlag, New York-Heidelberg, 1972).


\bibitem{LZ}I. Lucardesi, D. Zucco, On Blaschke-Santal\'o diagrams for the torsional rigidity and the first Dirichlet eigenvalue, Annali di Matematica Pura ed Applicata https://doi.org/10.1007/s10231--021--01113--6.

\bibitem{PSZ}
G. P\'{o}lya, G. Szeg\"{o}, Isoperimetric Inequalities in Mathematical Physics, Ann. of Math. Stud. 27, Princeton University Press, Princeton, 1951.

\bibitem{S} P. R. Scott, P. W. Awyong, Inequalities for convex sets. J. Inequal. Pure Appl. Math. \textbf{1} (2000), Article 6.

\bibitem{T} G. J. Tee, Surface area and capacity of ellipsoids in $n$ dimensions,
New Zealand J. Math. \textbf{34} (2005), 165--198.
\end{thebibliography}
\end{document}